\documentclass{article}
\usepackage[utf8]{inputenc}
\usepackage[margin= 1.75 cm]{geometry}

\usepackage{amsthm}
\usepackage{enumitem}
\usepackage{amsmath}
\usepackage{amssymb}
\usepackage{setspace}
\usepackage{mathtools}
\usepackage{verbatim}
\usepackage{csquotes}
\usepackage{graphicx}
\usepackage[hidelinks]{hyperref}
\usepackage{thm-restate}
\usepackage{cleveref}
\usepackage{graphicx}
\usepackage{mathrsfs}
\usepackage{enumitem}
\usepackage{framed}
\usepackage{subcaption}
\usepackage{crossreftools}
\usepackage{floatrow}
\usepackage[T1]{fontenc}
\usepackage{appendix}
\usepackage{xurl}
\usepackage{bbm}

\theoremstyle{plain}

\newtheorem*{thm*}{Theorem}
\newtheorem{theorem}{Theorem}[section]
\Crefname{theorem}{Theorem}{Theorems}

\newtheorem*{lem*}{Lemma}
\newtheorem{lemma}[theorem]{Lemma}
\Crefname{lemma}{Lemma}{Lemmas}

\newtheorem*{claim*}{Claim}

\crefname{claim}{Claim}{Claims}
\Crefname{claim}{Claim}{Claims}

\newtheorem{prop}[theorem]{Proposition}
\Crefname{prop}{Proposition}{Propositions}

\newtheorem{corollary}[theorem]{Corollary}
\crefname{corollary}{Corollary}{Corollaries}

\newtheorem{conj}[theorem]{Conjecture}
\crefname{conj}{Conjecture}{Conjectures}

\newtheorem*{conj*}{Conjecture}

\Crefname{qn}{Question}{Questions}

\Crefname{obs}{Observation}{Observations}

\Crefname{ex}{Example}{Examples}

\theoremstyle{definition}

\Crefname{prob}{Problem}{Problems}

\newtheorem{defn}[theorem]{Definition}
\Crefname{defn}{Definition}{Definitions}

\newtheorem*{defn*}{Definition}

\theoremstyle{remark}

\renewenvironment{proof}[1][]{\begin{trivlist}
\item[\hspace{\labelsep}{\bf\noindent Proof#1.\/}] }{\qed\end{trivlist}}

\newcommand{\eps}{\varepsilon}

\renewcommand{\P}{\mathbb{P}}

\newcommand{\G}{\mathcal{G}}

\DeclareMathOperator{\bin}{Bin}

\expandafter\def\expandafter\normalsize\expandafter{%
    \normalsize
    \setlength\abovedisplayskip{8pt}
    \setlength\belowdisplayskip{8pt}
    \setlength\abovedisplayshortskip{4pt}
    \setlength\belowdisplayshortskip{4pt}
}

\usepackage[square,sort,comma,numbers]{natbib}
 \setlist[itemize]{leftmargin=*}

\makeatletter
\newcommand{\optionaldesc}[2]{%
  \phantomsection
  #1\protected@edef\@currentlabel{#1}\label{#2}%
}
\makeatother

\title{Regular subgraphs at every density}

\author{Debsoumya Chakraborti\thanks{
Mathematics Institute, University of Warwick, Coventry, United Kingdom. Research supported by the European Research Council (ERC) under the European Union Horizon 2020 research and innovation programme (grant agreement No.\ 947978). 
E-mail: {\tt
\{debsoumya.chakraborti, richard.montgomery\}@warwick.ac.uk}.} 
\and Oliver Janzer\thanks{Institute of Mathematics, EPFL, Lausanne, Switzerland. Email: {\tt oliver.janzer@epfl.ch}}
\and Abhishek Methuku\thanks{Department of Mathematics, University of Illinois at Urbana–Champaign, Urbana, IL, USA. Research supported by the UIUC Campus Research Board Award RB25050.
Email: {\tt
abhishekmethuku@gmail.com}} 
\and Richard Montgomery\footnotemark[1]
}
\date{}

\begin{document}

\maketitle

\begin{abstract}
In 1975, Erd\H os and Sauer asked to estimate, for any constant $r$, the maximum number of edges an $n$-vertex graph can have without containing an $r$-regular subgraph. In a recent breakthrough, Janzer and Sudakov proved that any $n$-vertex graph with no $r$-regular subgraph has at most $C_r n \log \log n$ edges, matching an earlier lower bound by Pyber, R\"odl and Szemer\'edi and thereby resolving the Erd\H os--Sauer problem up to a constant depending on~$r$. We prove that every $n$-vertex graph without an $r$-regular subgraph has at most $Cr^2 n \log \log n$ edges. This bound is tight up to the value of $C$ for $n\geq n_0(r)$ and hence resolves the Erd\H os--Sauer problem up to an absolute constant.

Moreover, we obtain similarly tight results for the whole range of possible values of $r$ (i.e., not just when $r$ is a constant), apart from a small error term at a transition point near $r\approx \log n$, where, perhaps surprisingly, the answer changes. More specifically, we show that every $n$-vertex graph with average degree at least $\min(Cr\log(n/r),Cr^2 \log\log n)$ contains an $r$-regular subgraph. The bound $Cr\log(n/r)$ is tight for $r\geq \log n$, while the bound $Cr^2 \log \log n$ is tight for $r<(\log n)^{1-\Omega(1)}$. These results resolve a problem of R\"odl and Wysocka from 1997 for almost all values of $r$.

To find regular subgraphs efficiently, we replace and improve substantial parts of the Janzer--Sudakov framework for the Erd\H{o}s--Sauer problem. Our strategies involve the algebraic techniques of Alon, Friedland and Kalai, the recent breakthroughs on the sunflower conjecture, techniques to find almost-regular subgraphs developed from Pyber's work, and, crucially, a novel random process that efficiently finds a very nearly regular subgraph in any almost-regular graph.
A key step in our proof uses this novel random process to show that every $K$-almost-regular graph with average degree~$d$ contains an $r$-regular subgraph for some $r=\Omega_K(d)$, which is of independent interest.
\end{abstract}

\section{Introduction}

Finding regular subgraphs with given degree is a fundamental problem in Graph Theory.
Many important theorems in Graph Theory assume that the graph is regular, or are much easier to prove in that case. Hence, developing tools for finding regular subgraphs is of great importance, and such tools can be applied in many different areas of Combinatorics (see, e.g., \cite{janzer2023resolution} for some of these applications).
 
In 1975, Erd\H os and Sauer \cite{Erd75} asked what is the maximum possible average degree that an $n$-vertex graph can have without containing an $r$-regular subgraph. This question received a lot of interest in the case where $r$ is a constant, and was reiterated in many places such as Bollob\'as's book on Extremal Graph Theory \cite{Bol78} and the book of Bondy and Murty \cite{BM76}, while it was also one of Erd\H os's favourite open problems~\cite{Erd81}.

In the early 80s, a very influential algebraic technique for finding regular subgraphs with given degree was developed by Alon, Friedland and Kalai \cite{AFK84,AFK84short}. Informally speaking, they proved that a graph which is very close to being regular (in the sense that its average degree is very close to its maximum degree) contains a regular subgraph of large degree (see Theorem \ref{thm:AFK} and Corollary \ref{cor:AFK} below for the formal statements). This result, together with new ideas, was used by Pyber \cite{pyber1985regular} in 1985 to obtain the first strong bounds for the minimum average degree condition which implies the existence of an $r$-regular graph.

\begin{theorem}[Pyber \cite{pyber1985regular}] \label{thm:pyber}
    For any positive integers $r$ and $n$, every $n$-vertex graph with average degree greater than $32r^2\log n$ contains an $r$-regular subgraph.
\end{theorem}

When $r$ is constant Theorem~\ref{thm:pyber} is tight up to the factor of $\log n$, but as $r$ increases the bound is increasingly far from optimal (as discussed later). 
Chv\'atal had conjectured that some average degree condition depending only on $r$ should be sufficient to imply the existence of an $r$-regular graph (see~\cite{BM76}), but interestingly this turns out to be false. Indeed, in 1995, Pyber, R\"odl and Szemer\'edi \cite{PRSz95} constructed examples demonstrating that an $n$-vertex graph can have $\Omega(n\log\log n)$ edges, yet no $r$-regular subgraph, as follows.

\begin{theorem}[Pyber--R\"odl--Szemer\'edi \cite{PRSz95}]\label{thm:PRSorigconstruction}
There is some absolute constant $c>0$ such that, for each $n$, there is an $n$-vertex graph with average degree at least $c\log\log n$ which does not contain an $r$-regular subgraph for any $r\geq 3$.
\end{theorem}

In light of Theorem~\ref{thm:pyber}, in 1997, R\"odl and Wysocka \cite{RW97} posed the problem of estimating the smallest $d=d(r,n)$ such that every $n$-vertex graph with average degree at least $d$ contains an $r$-regular subgraph. They proved, for every $\gamma\geq 0$, that every $n$-vertex graph with average degree at least $\gamma n$ contains an $r$-regular subgraph for some $r\geq C\gamma^3 n$, where $C$ is some absolute constant. This directly improved a previous result of Pyber, R\"odl and Szemer\'edi \cite{PRSz95}, and also improves upon Pyber's bound in Theorem~\ref{thm:pyber} when the host graph is fairly dense (more precisely, when its average degree is greater than roughly $n^{4/5}$).

Despite continued related work, such as that of Bollob\'as, Kim and Verstra\"ete \cite{BKV06} on the threshold for a random graph to contain an $r$-regular subgraph, and many papers on finding regular subgraphs in hypergraphs (see, e.g., \cite{MV09,DHL+,KK14,Kim16,HK18,JST22+}), the bounds $O_r(\log n)$ and $\Omega(\log \log n)$ for the Erd\H os--Sauer problem have not been improved until very recently, when Janzer and Sudakov \cite{janzer2023resolution} proved that the answer is $O_r(\log \log n)$, as follows.

\begin{theorem}[Janzer--Sudakov \cite{janzer2023resolution}] \label{thm:JSES}
    For every positive integer $r$ there exists some $C_r$ such that any $n$-vertex graph with average degree at least $C_r\log \log n$ contains an $r$-regular subgraph.
\end{theorem}

In this paper we resolve the problem of R\"odl and Wysocka for almost all values of $r$ and $n$, and in particular, settle the Erd\H os--Sauer problem up to an absolute constant factor. Perhaps surprisingly, the answer changes drastically around $r\approx \log n$. We prove two upper bounds, the first of which is tight when $r$ is small, and the second of which is tight when $r$ is large.

\begin{theorem} \label{thm:general ES}
	There exists some $C$ such that for all positive integers $r,n\geq 3$, every $n$-vertex graph with average degree at least $Cr^2 \log \log n$ contains an $r$-regular subgraph. 
\end{theorem}

\begin{theorem}\label{thm:reg}
There exists some $C$ such that for all positive integers $r,n$ with $r\leq n/2$, every $n$-vertex graph with average degree at least $Cr\log(n/r)$ contains an $r$-regular subgraph.
\end{theorem}

We prove matching lower bounds for both results, showing that Theorem \ref{thm:general ES} is tight for $r<(\log n)^{1-\Omega(1)}$, whereas Theorem \ref{thm:reg} is tight for $r\geq \log n$. In particular, to resolve the Erd\H os--Sauer problem up to an absolute constant, our Theorem \ref{thm:general ES} improves the value of $C_r\approx r^{16}$ used in the proof of Theorem \ref{thm:JSES}.

\begin{prop} \label{prop:constructiondsmall}
	There is some $c>0$ such that for all positive integers $r$ and $n$ with $3\leq r\leq \frac{1}{2}\log n$, there exists an $n$-vertex graph with average degree at least $cr^2\log(\frac{\log n}{r})$ which does not contain an $r$-regular subgraph. In particular, Theorem \ref{thm:general ES} is tight up to the value of $C$ provided that $n$ is sufficiently large compared to $r$.
\end{prop}

\begin{restatable}{prop}{lowerbounddlarge}
 \label{prop:lower bound d large}
    There is some $c>0$ such that for all positive integers $r,n\geq 2$ with $\frac{1}{2}\log n\leq r\leq n/100$, there exists an $n$-vertex graph with average degree at least $cr\log(n/r)$ which does not have an $r$-regular subgraph.
\end{restatable}

We prove Propositions \ref{prop:constructiondsmall} and \ref{prop:lower bound d large} using careful modifications of the construction of Pyber, R\"odl and Szemer\'edi \cite{PRSz95}.
We note that a similar, but slightly less general result to Proposition \ref{prop:lower bound d large} was obtained by Buci{\'c}, Kwan, Pokrovskiy, Sudakov, Tran, and Wagner~\cite{bucic2020nearly}.

Recall that $d(r,n)$ is defined as the smallest $d$ such that every $n$-vertex graph with average degree at least $d$ contains an $r$-regular subgraph. Combining Theorems \ref{thm:general ES} and \ref{thm:reg} and  Propositions \ref{prop:constructiondsmall} and \ref{prop:lower bound d large}, we obtain the following interesting `phase transition' phenomenon:
\begin{itemize}
    \item for $r<(\log n)^{1-\Omega(1)}$, we have $d(r,n)=\Theta(r^2\log \log n)$, whereas
    \item for $r\geq \log n$, we have $d(r,n)=\Theta(r \log(n/r))$.
\end{itemize}

\subsection{Almost-regular graphs}

In the process of proving Theorems \ref{thm:general ES} and \ref{thm:reg}, we will consider various notions of `near-regularity'. Proving results about nearly-regular subgraphs is interesting in its own right, and has already attracted a lot of interest. The notion that was the most extensively studied is `almost-regularity'.

\begin{defn}
    A graph is called \emph{$K$-almost-regular} if the ratio between the maximum degree and the minimum degree is at most $K$.
\end{defn}

An important reason why this notion is useful is the celebrated `regularisation lemma' of Erd\H os and Simonovits \cite{ES70} from 1970. This result states that any $n$-vertex graph with at least $n^{1+\alpha}$ edges contains a $K$-almost-regular subgraph with $m$ vertices and at least $\frac{2}{5}m^{1+\alpha}$ edges for some $K=K(\alpha)$ and $m=\omega(1)$. This result has since become one of the most widely used tools in Tur\'an-type problems, as it allows one to reduce such problems to the easier case where the host graph is almost-regular. This demonstrates the value of results that find `dense' regular or nearly-regular subgraphs in general graphs, which are the subject of this paper.

Another useful result, which will be the starting point in our proof of Theorem \ref{thm:reg}, allows one to find an almost-regular subgraph in an arbitrary graph at the cost of a $\log n$ factor loss in the degree. The following result was proved in \cite{ARS+17} and in \cite{bucic2020nearly}, using a variant of Pyber's proof~\cite{pyber1985regular} of Theorem \ref{thm:pyber}.

\begin{lemma}[Alon et al.\ \cite{ARS+17} and Buci\'c et al.\ \cite{bucic2020nearly}] \label{lem:almostreg} Every $n$-vertex graph with average degree $d$ contains a $4$-almost-regular subgraph with average degree at least $\frac{d}{100\log n}$.    
\end{lemma}

We will in fact need a slightly stronger variant of Lemma \ref{lem:almostreg}, which we prove using a very similar approach (see Section~\ref{sec:almost-regular}).

\begin{prop} \label{prop:almostreg sharp}
    Every $n$-vertex graph with average degree $d$ contains a $4$-almost-regular subgraph with average degree at least $\frac{d}{100\log(32n/d)}$.  
\end{prop}

A very important result of Alon, Friedland and Kalai (as mentioned earlier) allows one to find regular subgraphs in graphs that are already very close to being regular.

\begin{theorem}[Alon--Friedland--Kalai \cite{AFK84}] \label{thm:AFK}
    If $q$ is a prime power, $q\geq r$, and $q\equiv r \textrm{ (mod 2)}$, then every graph with maximum degree $\Delta\geq 2q-2$ and average degree $d>\frac{2q-2}{2q-1}\cdot (\Delta+1)$ contains an $r$-regular subgraph.
\end{theorem}

Often this result is used via the following corollary\footnote{To see how \Cref{cor:AFK} follows from \Cref{thm:AFK}, note that one may choose $q$ to be a power of $2$ such that $\frac{d}{10\lambda}\leq q<\frac{d}{5\lambda}$ and apply \Cref{thm:AFK} with $r=q$.}.

\begin{corollary}[Alon--Friedland--Kalai] \label{cor:AFK}
    Let $\lambda\geq 1$. Then every graph with average degree $d$ and maximum degree at most $d+\lambda$ contains an $r$-regular subgraph for some $r\geq \frac{d}{10\lambda}$.
\end{corollary}

We prove that almost-regular graphs contain regular subgraphs with essentially the same degree, which is a significant strengthening of Corollary \ref{cor:AFK} as the condition on the maximum degree is much less restrictive.

\begin{theorem} \label{thm:regular in almost regular}
    For each $\lambda\geq 1$, there exists some $c=c(\lambda)>0$ such that every graph with average degree $d$ and maximum degree at most $\lambda d$ contains an $r$-regular subgraph for all $r\leq cd$.
\end{theorem}

This also strengthens a result of Janzer, Sudakov and Tomon \cite{JST22+}, which states that every $n$-vertex $\lambda$-almost-regular graph with average degree $d$ contains an $r$-regular subgraph for all $r\leq c(\lambda)d/\log n$.

Observe that Theorem \ref{thm:regular in almost regular}, combined with Proposition \ref{prop:almostreg sharp}, implies Theorem \ref{thm:reg} (see Section \ref{sec:almost-regular} for a formal proof). Theorem \ref{thm:regular in almost regular} will also play a key role in our proof of Theorem \ref{thm:general ES}.

Furthermore, we can use Theorem \ref{thm:regular in almost regular} to obtain a strengthening of the celebrated Erd\H os--Simonovits regularization lemma mentioned above. Before we state our result, we recall a slight variant of the Erd\H os--Simonovits regularization lemma that is more convenient in applications. This variant is due to Conlon, Janzer and Lee (for an earlier variant, see Jiang and Seiver~\cite{jiang2012turan}).

\begin{lemma}[Conlon--Janzer--Lee \cite{conlon2021more}]
    Let $\eps,c$ be positive reals, where $\eps < 1$. Let $n$ be a positive integer that is sufficiently large as a function of $\eps$ and $c$. Let $G$ be a graph on $n$ vertices with $e(G) \geq cn^{1+\eps}$. Then $G$ contains a $K$-almost-regular subgraph $G'$ on $q \geq n^{\frac{\eps-\eps^2}{4+4\eps}}$ vertices such that $e(G') \geq \frac{2c}{5}q^{1+\eps}$ and $K = 20\cdot 2^{\frac{1}{\eps^2}+1}$.
\end{lemma}

Using Theorem \ref{thm:regular in almost regular}, we can replace the $K$-almost-regular subgraph $G'$ by a fully regular subgraph which has essentially the same properties, and the average degree drops by at most a constant factor that depends on $\eps$.

\begin{theorem} \label{thm:Erdos Simonovits strengthening}
    For any positive real $\eps<1$, there exists some $\beta>0$ such that the following holds. Let $c>0$ and let $n$ be a positive integer that is sufficiently large as a function of $\eps$ and $c$. Let $G$ be a graph on $n$ vertices with $e(G)\geq cn^{1+\eps}$. Then $G$ contains a regular subgraph $H$ on $m\geq n^{\beta}$ vertices such that $e(H)\geq \beta c m^{1+\eps}$.
\end{theorem}

Note that while Theorem \ref{thm:regular in almost regular} does not directly control the number of vertices in the regular subgraph, the lower bound $m=n^{\Omega_{\eps}(1)}$ on the number of vertices in $H$ follows from the lower bound for the degrees of $H$.
We expect Theorem~\ref{thm:Erdos Simonovits strengthening} to become a significant tool in Tur\'an-type problems.

\subsection{Preliminaries and notation} \label{sec:prelim}

\textbf{Concentration.} We will often use the following basic version of Chernoff's inequality for binomial random variables (see, e.g., \cite{alon2016probabilistic}).
\begin{lemma}[Chernoff's bound]\label{chernoff}
Let $n$ be an integer and $0\le \delta,p \le 1$. If $X \sim \bin(n,p),$ then, setting $\mu=\mathbb{E} [X]= np,$ we have
$$\P(X>(1+\delta) \mu) \le e^{-\delta^2\mu/2},\quad\quad\quad \text{ and }\quad\quad\quad \P(X<(1-\delta) \mu) \le e^{-\delta^2\mu/3}.$$
\end{lemma}

\textbf{Notation.} We use standard graph theoretic notation throughout the paper. In particular, for a graph or hypergraph~$G$, we denote by $d(G)$ its average degree, and by $\delta(G)$ and $\Delta(G)$ its minimum degree and maximum degree, respectively. We write $|G|$ and $e(G)$ for the number of vertices and edges in the graph $G$, respectively. All graphs are assumed to have a non-empty vertex set. Logarithms are to base $2$. For the sake of a cleaner presentation, we sometimes omit floor and ceiling 
signs when they are not crucial.

\textbf{Organization of the paper.} In Section \ref{sec:nearreg}, we prove Theorem \ref{thm:regular in almost regular}. In Section \ref{sec:tight ES}, we prove Theorem \ref{thm:general ES}. Proposition~\ref{prop:almostreg sharp} and Theorem \ref{thm:reg} are proved in Section \ref{sec:almost-regular}. The proofs of Propositions~\ref{prop:constructiondsmall} and \ref{prop:lower bound d large} are given in Section \ref{sec:lower bounds}.  In Section~\ref{sec:concluding remarks}, we give some concluding remarks and mention some open problems.

\section{Regular subgraphs in almost-regular graphs} \label{sec:nearreg}

In this section we prove Theorem \ref{thm:regular in almost regular}.
Our main novel contribution is the following `near-regularisation' result.

\begin{theorem} \label{thm:constant diff}
    There is an absolute constant $c>0$ such that the following holds for all $\lambda\geq 1$ and $d> 0$. Let $G$ be a graph in which $d\leq d_G(v)\leq \lambda d$ for each $v\in V(G)$. Then, for some $d'\geq cd/\lambda^3$, $G$ contains a subgraph $H$ with average degree $d'$ and maximum degree at most $d'+10^8$ such that $|H|\geq c|G|/\lambda^{10}$.
\end{theorem}

Theorem \ref{thm:constant diff} states that in an almost-regular graph we can find a subgraph with average degree only a constant factor smaller than the average degree of the original graph that is extremely close to being regular. We can then apply the algebraic regularization technique of Alon, Friedland and Kalai (Corollary \ref{cor:AFK}) to find a fully-regular subgraph. We note that while we do not need this for the proof of Theorem \ref{thm:regular in almost regular}, the nearly-regular subgraph given by Theorem \ref{thm:constant diff} has a large vertex set, which may be useful for other applications.

 The proof of Theorem \ref{thm:constant diff} builds on the ideas recently introduced by the current authors, in \cite{chakraborti2024edge}. The key idea is to take an almost-regular graph and show that a small random deletion of vertices and edges (with probability differing based on whether the vertices involved are low- or high-degree vertices) is likely to produce a subgraph which is slightly closer to regular and yet has a very similar average degree, before using this iteratively. This will allow us to obtain the desired nearly-regular subgraph in $G$ by finding a sequence $G=:G_0\supset G_1\supset G_2\supset \dots$ of subgraphs such that each $G_{i+1}$ is a `bit more regular' than $G_i$ but has similar average degree. We terminate the process once $G_i$ is sufficiently regular and let $H$ be this final subgraph. In \cite{chakraborti2024edge}, we controlled the minimum degree and the maximum degree of each $G_i$, and using those methods one can show (via the Lov\'asz local lemma) that with positive probability we end up with a graph $G_i$ in which the minimum degree and the maximum degree differ by at most $O(\log d)$. This, however, would not be sufficient to prove Theorem \ref{thm:regular in almost regular}, so here we take a slightly different approach. One important observation is that in order to apply the algebraic regularization technique of Alon, Friedland and Kalai, one does not need to control the minimum degree; controlling the average degree and the maximum degree is sufficient. This can be achieved by repeating the following process. In each step, we first apply the small random regularization step from \cite{chakraborti2024edge}. When the average degree and the maximum degree are already very close (closer than $\log d$, say), this leaves some outlier `high-degree' vertices, but not too many -- we can therefore delete some edges incident to these vertices to `push down' their degrees to the interval we want.

The exact implementation of our small regularization step (i.e.\ how we obtain $G_{i+1}$ from $G_i$) is given by the following lemma. After stating the lemma we will motivate the use of the `correction' functions $f$ and $f''$ and comment further on the narrowing interval in which the (corrected) degrees lie.

\begin{lemma} \label{lem:one step with aux}
    Let $\eps\leq 1/100$ such that $\eps d\geq 10^6$ and $\Delta\geq (1+10\eps)d$. Let $f:V(G)\rightarrow \mathbb{Z}_{\geq 0}$ be a function. Assume that $G$ is a graph such that $d\leq d_G(v)+f(v)\leq \Delta$ for every $v\in V(G)$. Then there exist some subgraph $G''\subset G$ and some function $f'':V(G'')\rightarrow \mathbb{Z}_{\geq 0}$ such that 
    \begin{itemize}
        \item $|G''|\geq (1-4\eps)|G|$,
        \item $\sum_{v\in V(G'')} f''(v)\leq \sum_{v\in V(G)} f(v)+|G|\exp(-\eps d/1000)$ and
        \item $(1-\frac{5}{4}\eps)d\leq d_{G''}(v)+f''(v)\leq (1-\frac{7}{4}\eps)\Delta$ for every $v\in V(G'')$.
    \end{itemize}
\end{lemma}

Compare the intervals in which $d_G(v)+f(v)$ and $d_{G''}(v)+f''(v)$ are bounded for each vertex $v$: as $(1-\frac{7}{4}\eps)/(1-\frac{5}{4}\eps)\leq (1-\eps/2)$, the ratio between the upper and lower bounds is decreasing. We think of the non-negative function $f$ as a small `correction' term such that all `corrected degrees' $d_G(v)+f(v)$ are between $d$ and $\Delta$. Note that if $\sum_{v\in V(G)}f(v)$ is small, then the average degree of such a graph $G$ cannot be much smaller than $d$, while the maximum degree of $G$ is at most $\Delta$. The lemma states that, at the cost of increasing $f(v)$ very slightly, we can find a subgraph for which this lower bound $d$ and upper bound $\Delta$ are closer to each other. This corresponds to a subgraph in which the average degree and the maximum degree are closer than they were in the original graph. Crucially, unlike the corresponding result in \cite{chakraborti2024edge}, this lemma can be iterated until the average degree and the maximum degree differ by an absolute constant.

\begin{proof}[ of Lemma \ref{lem:one step with aux}] Let $\gamma$ be such that $\Delta=(1+\gamma)d$, and note that $\gamma\geq 10\eps$. Let 
\[
U_L=\{v\in V(G):d_G(v)+f(v)\leq (1+\gamma/2)d\}\quad\quad \text{ and} \quad\quad U_H=\{v\in V(G):d_G(v)+f(v)> (1+\gamma/2)d\}.
\]

Let $G'$ be the random subgraph of $G$ obtained by
\begin{enumerate}[label=(\roman*)]
    \item  deleting edges within $U_H$ independently at random with probability $2\eps-\eps^2$, \label{item:between dense}
    \item deleting edges from $U_{L}$ to $U_{H}$ independently at random with probability $\eps$, and \label{item:sparse and dense}
    \item deleting vertices in $U_{L}$  independently at random with probability $\eps$. \label{item:delete vertices}
\end{enumerate}

Furthermore, if some $v\in U_L$ belongs to $V(G')$, let $f'(v)$ be a random variable distributed as $\textrm{Bin}(f(v),1-\eps)$, and for $v\in U_H$, let $f'(v)$ be a random variable distributed as $\textrm{Bin}(f(v),(1-\eps)^2)$.

Now note that if $v\in U_L$ and $uv\in E(G)$, then, conditional on $v\in V(G')$, the probability that $uv\in E(G')$ is precisely $1-\eps$. Indeed, if $u\in U_L$, then (conditional on $v\in V(G')$) $uv\in E(G')$ holds if and only if $u$ did not get deleted by \ref{item:delete vertices}, which has probability $1-\eps$, and if $u\in U_H$, then (conditional on $v\in V(G')$) $uv\in E(G')$ holds if and only if the edge $uv$ did not get deleted by \ref{item:sparse and dense}, which has probability $1-\eps$.

Hence, if $v\in U_{L}$, then, conditional on $v\in V(G')$, we have 
\begin{equation}
    d_{G'}(v)+f'(v)\sim \textrm{Bin}(d_G(v)+f(v),1-\eps), \label{eqn:distr UL}
\end{equation}
while if $v\in U_{H}$, then, for each $u\in N_G(v)$, whether or not $u$ is in $U_L$, the probability $uv\in E(G')$ is $(1-\eps)^2$, and thus 
\begin{equation}
    d_{G'}(v)+f'(v)\sim \textrm{Bin}(d_G(v)+f(v),(1-\eps)^2). \label{eqn:distr UH}
\end{equation}

For each $v\in V(G')$, let $g(v)$ be the smallest non-negative integer such that $d_{G'}(v)+f'(v)+g(v)\geq (1-\frac{5}{4}\eps)d$.
We now show that with positive probability we have

\begin{enumerate}[label=\textbf{(\alph*)}]
    \item $|G'|\geq (1-4\eps)|G|$,\label{prop1}
    \item \label{property:few large} the number of vertices $v\in V(G')$ with $d_{G'}(v)> (1-\frac{7}{4}\eps)\Delta$ is at most $4|G|\exp(-\eps \Delta/500)$, and\label{prop2}
    \item \label{prop3}\begin{equation}
    \sum_{v\in V(G')} g(v)\leq 4|G|\exp(-\eps d/100). \label{eqn:bound g}
\end{equation}
\end{enumerate}

To see why \ref{prop1} holds, note that
the expected number of vertices deleted by \ref{item:delete vertices} is at most $\eps |G|$, so, by Markov's inequality, with probability at least~$3/4$, we have $|G'|\geq (1-4\eps)|G|$.

For \ref{prop2}, note that if $v\in U_L$ (and $v\in V(G')$), then 
\[
d_{G'}(v)+f'(v)\leq d_G(v)+f(v)\leq (1+\gamma/2)d<\left(1-\frac{7}{4}\eps\right)(1+\gamma)d= \left(1-\frac{7}{4}\eps\right)\Delta.
\]

Moreover, if $v\in U_H$, then by (\ref{eqn:distr UH})
\begin{align*}
    \mathbb{P}\left(d_{G'}(v)> \Big(1-\frac{7}{4}\eps\Big)\Delta \right)
    &\leq \mathbb{P}\left(d_{G'}(v)+f'(v)> \Big(1-\frac{7}{4}\eps\Big)\Delta \right) \\
    &\leq \mathbb{P}\left(\bin(\Delta,(1-\eps)^2)> \Big(1-\frac{7}{4}\eps\Big)\Delta\right) \\
    &= \mathbb{P}\left(\bin(\Delta,2\eps-\eps^2)< \frac{7}{4}\eps \Delta\right) \\
    &\leq \mathbb{P}\left(\bin\Big(\Delta,\frac{15}{8}\eps\Big)< \frac{7}{4}\eps \Delta\right)\leq e^{-(1/15)^2 \frac{15}{8} \eps \Delta/3}\leq \exp(-\eps \Delta/500),
\end{align*}
where the penultimate inequality follows from the Chernoff bound (Lemma~\ref{chernoff} applied with $\delta=1/15$), using that $\frac{7}{4}\eps \Delta=(1-\frac{1}{15})\cdot \mathbb{E}[\bin(\Delta,\frac{15}{8}\eps)]$.
Hence, by Markov's inequality, with probability at least~$3/4$, the number of vertices $v\in V(G')$ with $d_{G'}(v)> (1-\frac{7}{4}\eps)\Delta$ is at most $4|G|\exp(-\eps \Delta/500)$.

For \ref{prop3}, note that, if $v\in U_L$, then, by (\ref{eqn:distr UL}), for any positive integer $t$,
\begin{align*}
    \mathbb{P}\left(d_{G'}(v)+f'(v)\leq \Big\lceil \Big(1-\frac{5}{4}\eps\Big)d \Big\rceil-t \hspace{1mm} \Big| \hspace{1mm} v\in V(G')\right)
    &\leq \mathbb{P}\left(\bin(d,1-\eps)\leq \Big\lceil \Big(1-\frac{5}{4}\eps\Big)d \Big\rceil-t\right) \\
    &= \mathbb{P}\left(\bin(d,\eps)\geq \Big\lfloor\frac{5}{4}\eps  d \Big\rfloor +t\right) \\
    &\leq \mathbb{P}\left(\bin\left(d+\frac{5}{6}\eps^{-1}t,\eps\right)> \frac{6}{5}\eps  d  +t\right) \\
    &\leq \exp\left(-(t+\eps d)/60\right),
\end{align*}
where the last inequality follows from the Chernoff bound applied with $\delta=1/5$.

A similar calculation, using (\ref{eqn:distr UH}), shows that, for each $v\in U_H$ and any positive integer $t$, we have
\begin{align*}
\mathbb{P}\left(d_{G'}(v)+f'(v)\leq \Big\lceil \Big(1-\frac{5}{4}\eps\Big)d \Big\rceil-t\right)&\leq \mathbb{P}\left(\textrm{Bin}\Big(\Big(1+\frac{1}{2}\gamma\Big)d,(1-\eps)^2\Big) \leq \Big\lceil \Big(1-\frac{5}{4}\eps\Big)d \Big\rceil-t\right)\\
&\leq \mathbb{P}\left(\textrm{Bin}\Big( (1+5\eps)d,(1-\eps)^2\Big) \leq \Big\lceil \Big(1-\frac{5}{4}\eps\Big)d \Big\rceil-t\right)\\
&\leq \mathbb{P}\left(\textrm{Bin}\Big( (1+5\eps)d,2\eps-\eps^2\Big) \geq 6\eps d+t\right)\\
&\leq \mathbb{P}\left(\textrm{Bin}(2d,2\eps) \geq 6\eps d+t\right)\\
&\leq \exp\left(-(t+\eps d)/60\right).
\end{align*}

Hence, for each $v\in V(G)$ and positive integer $t$, the probability that $g(v)\geq t$ (conditional on $v\in V(G')$) is at most $\exp(-(t+\eps d)/60)$. It follows that, for each $v\in V(G)$, $$\mathbb{E}[g(v) \hspace{1mm} | \hspace{1mm} v\in V(G')]=\sum_{t=1}^{\infty} \mathbb{P}(g(v)\geq t \hspace{1mm} | \hspace{1mm} v\in V(G'))\leq \sum_{t=1}^{\infty} \exp(-(t+\eps d)/60)=\frac{e^{-1/60}}{1-e^{-1/60}}\exp(-\eps d/60)\leq \exp(-\eps d/100),$$ and therefore $\mathbb{E}[\sum_{v\in V(G')} g(v)]\leq |G|\exp(-\eps d/100)$. Hence, with probability at least $3/4$, we have $\sum_{v\in V(G')} g(v)\leq 4|G|\exp(-\eps d/100)$.

Thus, with positive probability, we have that \textbf{(a)}--\textbf{(c)} hold. We will show that, if we assume these three properties hold, then there exist suitable $G''$ and $f''$. Let us first define $G''$. The vertex set of $G''$ will be the same as that of $G'$. We obtain $G''$ by deleting a minimal collection of edges from $G'$ such that $d_{G''}(v)\leq (1-\frac{7}{4}\eps)\Delta$ holds for each $v\in V(G'')$. By property \ref{property:few large} and since $d_{G'}(v)\leq \Delta$ for all $v\in V(G')$, we need to delete at most $4|G|\exp(-\eps \Delta/500)\cdot 2\eps \Delta$ edges in total. Letting $h(v)=d_{G'}(v)-d_{G''}(v)$ be the number of deleted edges at vertex $v$, we have
\begin{equation}
    \sum_{v\in V(G'')} h(v)\leq 8|G|\exp(-\eps \Delta/500)\cdot 2\eps \Delta. \label{eqn:bound h}
\end{equation}
Moreover, $d_{G''}(v)+f'(v)+g(v)+h(v)=d_{G'}(v)+f'(v)+g(v)\geq (1-\frac{5}{4}\eps)d$ for all $v\in V(G'')$, by the definition of $g$ and $h$. So, since also $d_{G''}(v)\leq (1-\frac{7}{4}\eps)\Delta$ for all $v\in V(G'')$, there exists some non-negative integer valued function $f''(v)\leq f'(v)+g(v)+h(v)$ such that $(1-\frac{5}{4}\eps)d\leq d_{G''}(v)+f''(v)\leq (1-\frac{7}{4}\eps)\Delta$ for each $v\in V(G'')$. Then, by (\ref{eqn:bound g}) and (\ref{eqn:bound h}), we have
\begin{align*}
    \sum_{v\in V(G'')} f''(v)
    &\leq \sum_{v\in V(G'')} f'(v)+\sum_{v\in V(G'')} g(v)+\sum_{v\in V(G'')} h(v) \\
    &\leq \sum_{v\in V(G)} f(v)+4|G|\exp(-\eps d/100)+8|G|\exp(-\eps \Delta/500)\cdot 2\eps \Delta \\
    &\leq \sum_{v\in V(G)} f(v)+|G|\exp(-\eps d/1000),
\end{align*}
where the last inequality follows from $\eps d\geq 10^6$ and $\Delta\geq d$.
\end{proof}

The following lemma is proven by repeated applications of Lemma \ref{lem:one step with aux}.

\begin{lemma} \label{lem:min close to max with correction}
    There is an absolute constant $c>0$ such that the following holds for each $\lambda\geq 1$ and $d> 0$. Let $G$ be a graph with $d\leq d_G(v)\leq \lambda d$ for all $v\in V(G)$. Then there exist some subgraph $H\subset G$, some function $g:V(H)\rightarrow \mathbb{Z}_{\geq 0}$ and some $d'\geq cd/\lambda^3$ such that 
    \begin{itemize}
        \item $|H|\geq c|G|/\lambda^{10}$,
        \item $\sum_{v\in V(H)} g(v)\leq |H|$ and
        \item $d'\leq d_H(v)+g(v)\leq d'+10^{7}$ for every $v\in V(H)$.
    \end{itemize}
\end{lemma}

\begin{proof}
    We will show that $c=10^{-8}$ satisfies the statement in the lemma. Let $G$ be a graph with $d\leq d_G(v)\leq \lambda d$ for all $v\in V(G)$. If $d<10^{8} \lambda^3$, then the statement is trivial as we can let $d'=1$, let $H$ be a maximal matching in $G$ and let $g$ be the constant $0$ function on $V(H)$. The only non-trivial property to check is that $|H|\geq c|G|/\lambda^{10}$, which follows from $|G|d/2\leq |G|\delta(G)/2\leq e(G)\leq |H|\Delta(G)\leq |H|\lambda d$.
    Hence, we may assume that $d\geq 10^{8}\lambda^3$.
    
    We define $d_i$, $\Delta_i$ and $\eps_i$ recursively as follows. Let $d_0=d$ and $\Delta_0=\lambda d$. Having defined $d_i$ and $\Delta_i$, we terminate the process if $\Delta_i\leq d_i+10^7$. Otherwise, if $\Delta_i>d_i+10^7$, we choose $\eps_i$ to be maximal subject to $\eps_i\leq 1/100$ and $\Delta_i\geq (1+10\eps_i)d_i$. Then we let $d_{i+1}=(1-\frac{5}{4}\eps_i)d_i$ and $\Delta_{i+1}=(1-\frac{7}{4}\eps_i)\Delta_i$.
    Let $k$ be the minimal non-negative integer $i$ such that $\Delta_i\leq d_i+10^7$, and let $k=\infty$ if no such $i$ exists. We now prove three claims about the properties of $d_i$, $\Delta_i$ and $\eps_i$.

    \medskip

    \noindent \emph{Claim 1.} For all $0\leq j\leq k$, we have $d_j\geq d/\lambda^3$.

    \medskip

    \noindent \emph{Proof of Claim 1.} 
    Note that we have
    $$1\leq \frac{\Delta_j}{d_j}=\frac{\Delta_0}{d_0}\prod_{i=0}^{j-1}\frac{1-\frac{7}{4}\eps_i}{1-\frac{5}{4}\eps_i}\leq \frac{\Delta_0}{d_0}\prod_{i=0}^{j-1} (1-\eps_i/2)= \lambda \prod_{i=0}^{j-1} (1-\eps_i/2),$$
    so $\prod_{i=0}^{j-1} (1-\eps_i/2)\geq 1/\lambda$. Since $\eps_i\leq 1/100$ for each $i$, we have $1-{5\eps_i}/{4}\geq (1-\eps_i/2)^3$, and hence,
    $$d_j=d_0\prod_{i=0}^{j-1} (1-5\eps_i/4)\geq d_0\prod_{i=0}^{j-1} (1-\eps_i/2)^3\geq d_0/\lambda^3=d/\lambda^3,$$
    as desired.\hfill $\boxdot$

    \medskip

    \noindent \emph{Claim 2.} For all $0\leq j\leq k-1$, we have $\eps_j d_j\geq 10^6$.

    \medskip

    \noindent \emph{Proof of Claim 2.} By the definition of $\eps_j$, we have either $\eps_j=1/100$ or $\Delta_j=(1+10\eps_j)d_j$. In the former case, we have $\eps_j d_j=d_j/100\geq 10^6$, where we have used Claim 1 and the assumption that $d\geq 10^8\lambda^3$. In the latter case, observe that $\Delta_j>d_j+10^7$ by the definition of $k$, and therefore $10\eps_j d_j>10^7$, so $\eps_j d_j>10^6$.\hfill $\boxdot$

    \medskip

    \noindent \emph{Claim 3.} For all $0\leq j\leq k-2$, we have $\eps_{j+1} d_{j+1}\leq \frac{99}{100}\eps_j d_j$.

    \medskip

    \noindent \emph{Proof of Claim 3.} If $\eps_j=1/100$, then $\eps_{j+1}d_{j+1}=\eps_{j+1}(1-\frac{1}{80})d_j\leq \eps_j(1-\frac{1}{80})d_j<\frac{99}{100}\eps_j d_j$, as desired. Else, $\eps_i<1/100$ and $\Delta_j=(1+10\eps_j)d_j$. Now $\frac{\Delta_{j+1}}{d_{j+1}}=\frac{1-\frac{7}{4}\eps_j}{1-\frac{5}{4}\eps_j}\frac{\Delta_j}{d_j}\leq (1-\eps_j/2)\frac{\Delta_j}{d_j}=(1-\eps_j/2)(1+10\eps_j)<(1+\frac{99}{10}\eps_j)$. Hence, $\eps_{j+1}\leq \frac{99}{100}\eps_j$ and the claim follows since $d_{j+1}\leq d_j$. \hfill $\boxdot$

    \medskip

    Note that Claim 2 and Claim 3 together imply that $k$ is finite. We now apply Lemma~\ref{lem:one step with aux} iteratively, for the following claim.

    \medskip

    \noindent \emph{Claim 4.} For all $0\leq j\leq k$, there exist a subgraph $G_j\subset G$ and a function $f_j:V(G_j)\rightarrow \mathbb{Z}_{\geq 0}$ such that $d_j\leq d_{G_j}(v)+f_j(v)\leq \Delta_j$ for every $v\in V(G_j)$, $|G_j|\geq |G|\prod_{i=0}^{j-1} (1-4\eps_i)$, $f_0=0$ everywhere and, for all $1\leq j\leq k$, $\sum_{v\in V(G_j)} f_j(v)\leq |G_j|\exp(-\eps_{j-1} d_{j-1}/2000)$.

    \medskip

    \noindent \emph{Proof of Claim 4.} We prove the claim by induction on $j$. For $j=0$, we can take $G_0=G$ and let $f_0$ be the constant zero function.
    Now given suitable $G_j$ and $f_j$ for some $j\leq k-1$, we apply Lemma \ref{lem:one step with aux} with $\eps=\eps_j$, $d=d_j$, $\Delta=\Delta_j$, $f=f_j$ and $G=G_j$. The conditions $\eps_j\leq 1/100$ and $\Delta_j\geq (1+10\eps_j)d_j$ are satisfied by the definition of $\eps_j$, whereas the condition $\eps_j d_j\geq 10^6$ is satisfied by Claim 2. Now let $G_{j+1}=G''$ and $f_{j+1}=f''$ for the graph $G''$ and function $f''$ provided by Lemma \ref{lem:one step with aux}. Using the properties of $G_{j+1}$ and $f_{j+1}$ from Lemma~\ref{lem:one step with aux}, the induction hypothesis, $|G_{j+1}|\geq (1-4\eps_j)|G_j|\geq |G_j|/2$, Claim 2 and Claim 3, we have
    $$\sum_{v\in V(G_{j+1})} f_{j+1}(v)\leq |G_{j}|\exp(-\eps_{j-1}d_{j-1}/2000)+|G_j|\exp(-\eps_j d_j/1000)\leq |G_{j+1}|\exp(-\eps_{j}d_j/2000),$$
    where we define the term $|G_{j}|\exp(-\eps_{j-1}d_{j-1}/2000)$ to be zero if $j=0$.
    Hence, these choices satisfy the conditions in the claim. \hfill $\boxdot$

    Now let $H=G_k$, $g=f_k$ and $d'=d_k$. By Claim 1, we have $d'\geq d/\lambda^3\geq cd/\lambda^3$. By the definition of $k$, we have $\Delta_k\leq d_k+10^7$. As observed in the proof of Claim 1, we have $\prod_{i=0}^{k-1} (1-\eps_i/2)\geq 1/\lambda$, so $|H|=|G_k|\geq |G|\prod_{i=0}^{k-1} (1-4\eps_i)\geq |G|/\lambda^{10}\geq c|G|/\lambda^{10}$. Finally, $\sum_{v\in V(H)} g(v)=\sum_{v\in V(G_k)} f_k(v)\leq |H|$, using Claim 4 and Claim 2. 
\end{proof}

Lemma \ref{lem:min close to max with correction} readily implies Theorem \ref{thm:constant diff}.

\begin{proof}[ of Theorem \ref{thm:constant diff}] Using Lemma \ref{lem:min close to max with correction}, let $c_0\leq 1$ be a constant which satisfies the property in that lemma when replacing $c$.  
We will show that $c=c_0/2$ satisfies the properties of the theorem.
    This is trivial if $d\leq c^{-1}\lambda^3$, since then we can take $d'=1$ and let $H$ be a maximal matching in $G$. Assume therefore that $d>c^{-1}\lambda^3$. By Lemma \ref{lem:min close to max with correction}, there exists a subgraph $H\subset G$, some function $g:V(H)\rightarrow \mathbb{Z}_{\geq 0}$ and some $\tilde{d}\geq c_0d/\lambda^3$ such that $\tilde{d}\leq d_H(v)+g(v)\leq \tilde{d}+10^7$ for every $v\in V(H)$, $|H|\geq c_0|G|/\lambda^{10}$ and $\sum_{v\in V(H)} g(v)\leq |H|$. It follows that $H$ has average degree at least $\tilde{d}-1$ and maximum degree at most $\tilde{d}+10^7$. Since $\tilde{d}-1\geq c_0d/\lambda^3 -1=2cd/\lambda^3 -1 \geq cd/\lambda^3$, this subgraph $H$ satisfies the conditions of the theorem.
\end{proof}

Finally, here, we show how to deduce Theorem \ref{thm:regular in almost regular} from Theorem \ref{thm:constant diff} and Corollary \ref{cor:AFK}.

\begin{proof}[ of Theorem \ref{thm:regular in almost regular}] Using Theorem \ref{thm:constant diff}, let $c_0$ be a constant which satisfies the property in that theorem when replacing $c$.  
We will show that $c=c(\lambda)=10^{-10}c_0/(4\lambda)^3$ satisfies the conditions of the theorem. Let $G$ be a graph with average degree $d$ and maximum degree at most $\lambda d$. By standard results, $G$ has a bipartite subgraph $G'$ with average degree at least $d/2$, which, in turn, has a subgraph $G''$ with minimum degree at least $d/4$. Since $G''$ still has maximum degree at most $\lambda d$, we can apply Theorem \ref{thm:constant diff} for $G''$ with $d/4$ in place of $d$ and $4\lambda$ in place of $\lambda$ to find a subgraph $H$ and some $d'\geq c_0 (d/4)/(4\lambda)^3$ such that $H$ has average degree $d'$ and maximum degree at most $d'+10^8$. By Corollary~\ref{cor:AFK}, $H$ contains an $s$-regular subgraph $H'$ for some $s\geq \frac{d'}{10\cdot 10^8}\geq cd$. Since $H'$ is an $s$-regular bipartite graph (as it is a subgraph of $G'$), repeated applications of Hall's theorem yield an $r$-regular subgraph of $H'$ for all $1\leq r\leq s$. This finishes the proof of Theorem \ref{thm:regular in almost regular}.
\end{proof}

\section{Proof of Theorem \ref{thm:general ES}} \label{sec:tight ES}

In this section we prove Theorem \ref{thm:general ES}. Similarly to the strategy in Janzer and Sudakov's resolution of the Erd\H{o}s-Sauer problem \cite{janzer2023resolution}, to find an $r$-regular subgraph we will first reduce to two cases by finding a subgraph in which either \textbf{1)}~no vertices have very high degree or \textbf{2)} most of the edges contain an endvertex with high degree.  

In \textbf{case 1}, we can use the methods developed in \cite{janzer2023resolution} to find a sufficiently dense almost-regular subgraph. Then, crucially, using Theorem \ref{thm:regular in almost regular} that we proved in the previous section, we will be able to find an $r$-regular subgraph with no further loss. This is brought together to tackle \textbf{case 1} as Lemma~\ref{lem:almost bireg to reg}.

In \textbf{case 2}, more precisely, we find a bipartite subgraph $G'$ with parts $A$ and $B$ such that every $v\in B$ has degree precisely $r$ in $G'$, and the degrees of vertices in $A$ are at most $d^{1+\frac{1}{2r}}$ and on average at least $d$, for some $d$ which is quite a bit larger than $r$. We will then show that such a graph contains an $r$-regular subgraph (see Lemma \ref{lem:small jump implies k-regular} below). A similar result was proved in \cite{janzer2023resolution}, however with an extra assumption on the codegrees of $G'$. This assumption (which cannot be easily removed from the proof in \cite{janzer2023resolution}) results in a huge loss for the bounds if $r$ is not a constant, so we develop a new approach, using a similar strategy to that used to find regular subgraphs in hypergraphs by Janzer, Sudakov and Tomon \cite[Section~5]{JST22+}, and earlier by Kim~\cite{Kim16}.

We now turn to the details.
For dealing with \textbf{case 1}, we will need the following definition and lemma from \cite{janzer2023resolution}.

\begin{defn} \label{def:almost-biregular}
	A bipartite graph $G$ with vertex classes $A$ and $B$ is \emph{$(L,s)$-almost-biregular} if the following holds. For every $v\in B$, $d_G(v)=s$  and, writing $d=e(G)/|A|$, we have $d\geq s$ (or, equivalently, $|A|\leq |B|$) and $d_G(u)\leq Ld$ for every $u\in A$.
\end{defn}

\begin{lemma}[Janzer--Sudakov {\cite[Lemma 3.5]{janzer2023resolution}}] \label{lem:almost bireg to almost reg}
	Let $G$ be an $(L,s)$-almost-biregular graph for some $L\geq s\geq 2$. Then $G$ has a $64$-almost-regular subgraph with average degree at least $\frac{s}{16\log L}$.
\end{lemma}

Using our Theorem \ref{thm:regular in almost regular}, we can immediately improve Lemma \ref{lem:almost bireg to almost reg} to give a regular subgraph, as follows.

\begin{lemma} \label{lem:almost bireg to reg}
    There is an absolute constant $c>0$ such that the following holds. Let $G$ be an $(L,s)$-almost-biregular graph for some $L\geq s\geq 2$. Then $G$ has an $r$-regular subgraph for every $r\leq cs/\log L$.
\end{lemma}

For \textbf{case 2}, where we work in a subgraph where most of the edges have an endvertex of high degree, our key lemma will be as follows.

\begin{lemma} \label{lem:small jump implies k-regular}
	The following holds for all sufficiently large positive integers $r$. Let $G$ be a bipartite graph on at most $n$ vertices with parts $A$ and $B$ such that $d_G(v)=r$ for every $v\in B$, $d_G(u)\leq \Delta$ for every $u\in A$ and $e(G)=d|A|$, where $d\ge r^{3r}(\log n)^2$ and $\Delta\le d^{1+\frac{1}{2r}}$.
		Then $G$ contains an $r$-regular subgraph.
\end{lemma}

It will be crucial to turn this lemma into an essentially equivalent statement on regular subgraphs in multi-hypergraphs. Here, $d(\mathcal{G})$ and $\Delta(\mathcal{G})$ are, respectively, over $v\in V(\mathcal{G})$, the average and maximum number of edges in $\mathcal{G}$ containing $v$.

\begin{lemma} \label{lem:regular hypergraph}
	The following holds for all sufficiently large positive integers $r$. Let $\G$ be an $r$-uniform multi-hypergraph on $N$ vertices such that $\Delta(\G)\leq d(\G)^{1+\frac{1}{2r}}$ and $d(\G)\geq  r^{3r}(\log N)^2$. Then $\G$ contains an $r$-regular sub(multi)hypergraph. 
\end{lemma}

Let us first see why Lemma \ref{lem:regular hypergraph} implies Lemma \ref{lem:small jump implies k-regular}. Given a graph $G$ satisfying the conditions of Lemma \ref{lem:small jump implies k-regular}, let $\G$ be the $r$-uniform multi-hypergraph on vertex set $A$ whose hyperedges are $N_{G}(v)$ for all $v\in B$. Note that $\G$ has at most $n$ vertices, it has average degree $r\cdot e(\G)/|A|=r|B|/|A|=e(G)/|A|=d$ and maximum degree at most $\Delta$. Hence, by Lemma~\ref{lem:regular hypergraph}, $\G$ contains an $r$-regular sub(multi)hypergraph $\mathcal{H}$. The vertices of $\mathcal{H}$, together with the vertices $v\in B$ for which $N_G(v)$ is an edge in $\mathcal{H}$, induce an $r$-regular subgraph in $G$.

Hence, it will suffice to prove Lemma \ref{lem:regular hypergraph}.
In the proof of this lemma, we will use the recent breakthrough on the celebrated Sunflower conjecture \cite{ER60}. Given a positive integer $r$, a collection of $r$ sets is an \emph{$r$-sunflower} if there exists a (possibly empty) set $X$ (called the \emph{core}) such that the intersection of any two distinct elements in the collection is  $X$. It was proved by Alweiss, Lovett, Wu, and Zhang \cite{ALWZ20} that every $t$-uniform family of size at least $(\log t)^t\cdot (r\log\log t)^{O(t)}$ contains an $r$-sunflower, which was subsequently improved by Frankston, Kahn, Narayanan, and Park \cite{FKNP21}, Rao~\cite{Rao21} and Bell, Chueluecha and Warnke \cite{bell2021note}. We use the following version, proved by Rao \cite{Rao21}.

\begin{lemma}\label{lemma:sunflower}
	There is a constant $\alpha$ such that any family with  more than $(\alpha r\log (tr))^t$ sets of size $t$ contains an $r$-sunflower. 
\end{lemma}

The $r$-regular subhypergraph in Lemma \ref{lem:regular hypergraph} will come from an $r$-sunflower in a collection of large matchings in $\G$.  To bound the number of matchings from below, we will use the following simple lemma.

\begin{lemma}\label{lemma:num_matching}
	Let $\G$ be an $r$-uniform multi-hypergraph on $N$ vertices with $\Delta(\G)\leq \mu d(\G)$. Let $t\leq \frac{N}{2r^2\mu}+1$ be a positive integer. Then $\G$ contains at least $(\frac{e(\G)}{2t})^t$ matchings of size $t$. 
\end{lemma}

\begin{proof}
	
	We generate a matching of size $t$ with the following greedy procedure. Let $\G_1=\G$. In step $i$ ($1\leq i\leq t$), if $\G_i$ is already defined and $\G_i$ is nonempty, select an arbitrary hyperedge $e_i$ of $\G_i$, and let $\G_{i+1}$ be the hypergraph we get after removing $e_i$ and all hyperedges intersecting $e_i$ from $\G_{i}$. Clearly, $\{e_1,\dots,e_{i}\}$ is a matching. Note that every hyperedge of $\G$ intersects at most $r\Delta(\G)$ hyperedges, so $$e(\G_i)\geq e(\G)-(i-1)r\Delta(\G)\geq e(\G)-(t-1)r\Delta(\G)\geq e(\G)-\frac{N}{2r^2\mu}r\Delta(\G)\geq e(\G)-\frac{N}{2r^2\mu}r\mu d(\G)=e(\G)/2.$$
	
	Hence, at step $i$, there are at least $\frac{1}{2}e(\G)$ choices for the hyperedge $e_i$, so in total there are at least $(\frac{1}{2}e(\G))^t$ sequences $e_1,\dots,e_t$ such that $\{e_1,\dots,e_{t}\}$ is a matching. But, then, there are at least $\frac{1}{t!}(\frac{1}{2}e(\G))^t\geq (\frac{e(\G)}{2t})^{t}$ matchings of size $t$ in $\G$, as each matching corresponds to $t!$ such sequences.
\end{proof}

\begin{proof}[ of Lemma \ref{lem:regular hypergraph}]
	Let $\mu=d(\G)^{\frac{1}{2r}}$ and let $t=\left\lceil \frac{N}{2r^2\mu} \right\rceil$. By Lemma~\ref{lemma:num_matching}, $\G$ contains at least $M=(\frac{e(\G)}{2t})^t$ matchings of size $t$. Each such matching covers exactly $tr$ vertices, so by the pigeonhole principle there exists a set $U\subset V(\G)$ of size $tr$ such that at least 
	$M/\binom{N}{tr}$
	matchings of size $t$ span $U$. Using $\binom{a}{b}\leq (ea/b)^{b}$, $t\geq \frac{N}{2r^2\mu}$, $\mu=d(\G)^{\frac{1}{2r}}$ and $d(\G)\geq  r^{3r}(\log N)^2$, we have
	\begin{align*}
		M\Big/\binom{N}{tr}
		&\geq M\left(\frac{tr}{eN}\right)^{tr}=\left(\frac{e(\G)}{2t} \left(\frac{tr}{eN}\right)^r\right)^t\geq \left(\frac{r^2\mu e(\G)}{N} \left(\frac{1}{2er\mu}\right)^r\right)^t=\left(r\mu d(\G) \left(\frac{1}{2er\mu}\right)^r\right)^t \\
		&= \left(rd(\G)^{1+\frac{1}{2r}} \left(\frac{1}{2erd(\G)^{\frac{1}{2r}}}\right)^r\right)^t\geq \left(\frac{d(\G)^{1/2}}{(2er)^r}\right)^t >(2\alpha r \log N)^t\geq(\alpha r\log(tr))^{t},
	\end{align*}
    where $\alpha$ is a constant with the property in Lemma \ref{lemma:sunflower}.
	Let $\mathcal{F}$ be the family of sets whose ground set consists of all the edges of $\G$ in $U$ (note that multi-edges give rise to multiple elements in the ground set) and whose sets are the matchings spanning $U$. By the above equality, we have $|\mathcal{F}|> (\alpha r\log(tr))^{t}$, so Lemma~\ref{lemma:sunflower} guarantees that $\mathcal{F}$ contains an $r$-sunflower $S$. This is a collection of $r$
	matchings spanning $U$, such that every edge either belongs to all the matchings or only to one of them. Remove the edges contained in the core of $S$, and let the resulting set be $S'$. Then the edges in $S'$ form an $r$-regular sub(multi)hypergraph of $\G$, finishing the proof.
\end{proof}

We are now ready to prove Theorem \ref{thm:general ES}.

\begin{proof}[ of Theorem \ref{thm:general ES}]
    We may assume that $r$ is sufficiently large (e.g.\ because the case where $r$ is small follows from Theorem \ref{thm:JSES}). 
    Suppose then that $C$ is a sufficiently large absolute constant, that $n\geq 3$, and that $G$ is an $n$-vertex graph with average degree at least $Cr^2\log\log n$. Note that $G$ has a bipartite subgraph $G'$ with average degree at least $\frac{C}{2}r^2\log \log n$ and $G'$ has a subgraph $G''$ with minimum degree at least $\frac{C}{4}r^2\log \log n$. Let $A$ and $B$ be the parts of $G''$ such that $|A|\leq |B|$. Let $H$ be a spanning subgraph of $G''$ such that $d_H(v)=\frac{C}{4}r^2\log \log n$ for every $v\in B$.
    
    Let $t_0=100r\log \log n$ and let $\ell$ be the smallest non-negative integer such that $t_0(1+\frac{1}{3r})^{\ell}\geq \log n$. Note that $(1+\frac{1}{3r})^{6r\log \log n}\geq \exp(\log \log n)\geq \frac{\log n}{t_0}$, so $\ell\leq 6r\log \log n$. For each $1\leq i\leq \ell$, let $t_i=t_0(1+\frac{1}{3r})^{i}$. By definition, $t_\ell\geq \log n$.
	
	Let $A_0=\{u\in A: d_H(u)\leq 2^{t_0}\}$ and for each $1\leq i\leq \ell$, let $A_i=\{u\in A: 2^{t_{i-1}}<d_H(u)\leq 2^{t_i}\}$. Clearly, these sets partition~$A$. Hence, for every $v\in B$, either $v$ has at least $d_H(v)/2= \frac{C}{8}r^2\log \log n$ neighbours (in the graph $H$) in $A_0$, or $\ell>0$ and there is some $1\leq i\leq \ell$ such that $v$ has at least $d_H(v)/(2\ell)\geq \frac{C}{48}r\geq r$ neighbours in $A_i$ (where the last inequality holds as $C$ is sufficiently large).
	
	Therefore, at least one of the following two cases must occur.
	
	\noindent \textbf{Case 1.} There are at least $|B|/2$ vertices $v\in B$ which have at least $\frac{C}{8}r^2\log \log n$ neighbours in $A_0$.
	
	\noindent \textbf{Case 2.} There exist some $1\leq i\leq \ell$ and at least $|B|/(2\ell)$ vertices $v\in B$ which have at least $r$ neighbours in $A_i$.
	
	\sloppy In Case 1, let $B'\subset B$ be a set of size at least $|B|/2$ such that every $v\in B'$ has at least $\frac{C}{8}r^2\log \log n$ neighbours in $A_0$. For technical reasons, let us take a random subset $A_0'\subset A_0$ of size $|A_0|/3$. With positive probability, there is a set $B''\subset B'$ of at least $2|B'|/3$ vertices which all have at least $\frac{C}{100}r^2\log \log n$ neighbours in $A_0'$. Let $H'$ be a spanning subgraph of $H[A_0'\cup B'']$ obtained by keeping precisely $\frac{C}{100}r^2\log \log n$ edges from each vertex in $B''$. Now note that $|A_0'|=|A_0|/3\leq |A|/3\leq |B|/3\leq 2|B'|/3\leq |B''|$. Moreover, $d_{H'}(u)\leq d_H(u)\leq 2^{t_0}$ for every $u\in A_0'$. This implies that $H'$ is $(L,s)$-almost biregular for $L=2^{t_0}=2^{100r\log \log n}$ and $s=\frac{C}{100}r^2\log \log n$. Note that $L\geq s$. Hence, by Lemma~\ref{lem:almost bireg to reg}, $H'$ has an $r$-regular subgraph since $C$ is sufficiently large.
	
	In Case 2, let us choose some $1\leq i\leq \ell$ and a set $B'\subset B$ of size at least $|B|/(2\ell)$ such that, for every $v\in B'$, $v$ has at least $r$ neighbours in $A_i$. Let $H'$ be a subgraph of $H[A_i\cup B']$ obtained by keeping precisely $r$ edges incident to each $v\in B'$. We will apply Lemma \ref{lem:small jump implies k-regular} to this graph. Note that $\ell\geq 1$ implies that $r\leq \log n$. Let $\Delta=2^{t_i}$ and let $\delta=\Delta^{1/(1+\frac{1}{2r})}$. Clearly $\delta\geq \Delta^{1/2}\geq 2^{t_0/2}= 2^{50r\log \log n}>r^{3r}(\log n)^2$, where the last inequality follows from $r\leq \log n$. Note that for any $u\in A_i$, we have $d_{H'}(u)\leq d_{H}(u)\leq 2^{t_i}= \Delta$. Using that $|B|\cdot \frac{C}{4}r^2\log \log n= e(H)\geq |A_i|\cdot 2^{t_{i-1}}$ and that $\ell\leq 6r\log\log n$, we get
	\begin{align*}
	    e(H')
	    &=|B'|r\geq |B|r/(2\ell)=\frac{e(H)}{\frac{C}{2}\ell r \log \log n}\geq \frac{e(H)}{3Cr^2(\log \log n)^2}\geq \frac{|A_i|2^{t_{i-1}}}{3Cr^2(\log \log n)^2}.
	\end{align*}
	Note that 
 $$\delta=\Delta^{1/(1+\frac{1}{2r})}=2^{t_i/(1+\frac{1}{2r})}=2^{t_{i-1}(1+\frac{1}{3r})/(1+\frac{1}{2r})}\leq 2^{t_{i-1}-\frac{t_{i-1}}{10r}}\leq \frac{2^{t_{i-1}}}{3Cr^2(\log \log n)^2},$$
 where the last inequality holds since $\frac{t_{i-1}}{10r}\geq \frac{t_0}{10r}=10\log \log n$, $r\leq \log n$ and $r$ is sufficiently large.
	Hence, $e(H')\geq \delta|A_i|$.
	
	Thus, we can apply Lemma \ref{lem:small jump implies k-regular} to the graph $H'$ (with $d=e(H')/|A_i|\geq \delta$) and get an $r$-regular subgraph.
\end{proof}

\section{Proof of Theorem \ref{thm:reg}} \label{sec:almost-regular}

In this section we first prove Proposition \ref{prop:almostreg sharp} using a method of Pyber~\cite{pyber1985regular} developed similarly as was done in \cite{ARS+17,bucic2020nearly} for the slightly weaker Lemma~\ref{lem:almostreg}. We then deduce Theorem \ref{thm:reg} from it, using our result on regular subgraphs in almost-regular graphs (Theorem~\ref{thm:regular in almost regular}).
As is well-known, and as we have used in the deduction of Theorem~\ref{thm:regular in almost regular}, every graph with average degree $d$ has a bipartite subgraph with minimum degree at least $d/4$. Hence, Proposition \ref{prop:almostreg sharp} follows from the following result.

\begin{lemma}
    Let $G$ be a bipartite graph on at most $n$ vertices with minimum degree at least $d\geq 1$. Then $G$ has a $4$-almost-regular subgraph with average degree at least $\frac{d}{25\log(8n/d)}$.
\end{lemma}

\begin{proof}
    Let $A$ and $B$ be the parts of an arbitrary bipartition of $G$. Without loss of generality, assume that $|A|\geq |B|$. Let $A_1\subset A$ be a minimal non-empty subset with $|N_G(A_1)|\leq |A_1|$. This is well-defined since $|N_G(A)|\leq |B|\leq |A|$. Let $B_1=N_G(A_1)$. By the minimality of $A_1$ and as $d\geq 1$, we have $|B_1|=|A_1|$. Again by the minimality of $A_1$, for each non-empty $S\subsetneq A_1$, we have $|N_G(S)|>|S|$, so by Hall's theorem, there is a perfect matching $M_1$ between $A_1$ and $B_1$ in $G$. Let $G_0=G$ and $G_1=G[A_1,B_1]-M_1$.

    We now iteratively find sequences of vertex sets $A_1\supset A_2\supset \ldots \supset A_d$ and $B_1\supset B_2\supset \ldots \supset B_d$ along with disjoint matchings $M_1,M_2,\ldots, M_d$ such that, for each $i\in [d]$, $|A_i|=|B_i|$, $M_i$ is a perfect matching in the graph $G[A_i,B_i]$, and, setting $G_i=G[A_i,B_i]-M_1-\ldots -M_{i}$, $d_{G_i}(u)\geq d-i$ for each $u\in A_i$. Suppose then that  $1\leq i<d$ is an integer and we have found $A_i\subset A$ and $B_i\subset B$ with $|A_i|=|B_i|$ and a bipartite graph $G_i$ between them such that $d_{G_i}(u)\geq d-i$ for each $u\in A_i$. We define $A_{i+1}$, $B_{i+1}$ and $G_{i+1}$ as follows. Let $A_{i+1}\subset A_i$ be a minimal non-empty subset so that $|N_{G_i}(A_{i+1})|\leq |A_{i+1}|$. This is well-defined since $|N_{G_i}(A_i)|\leq |B_i|=|A_i|$. Let $B_{i+1}=N_{G_i}(A_{i+1})$. By the minimality of $A_{i+1}$ and using that $d_{G_i}(u)\geq d-i>0$ for every $u\in A_i$, we have $|B_{i+1}|=|A_{i+1}|$. Again by the minimality of $A_{i+1}$, for each non-empty $S\subsetneq A_i$, we have $|N_{G_i}(S)|>|S|$, so by Hall's theorem, there is a perfect matching $M_{i+1}$ between $A_{i+1}$ and $B_{i+1}$ in $G_i$. Let $G_{i+1}=G_i[A_{i+1},B_{i+1}]-M_{i+1}$. Note that, as $B_{i+1}=N_{G_i}(A_{i+1})$, each $u\in A_{i+1}$ has degree at least $d-i-1$ in $G_{i+1}$, so that we can continue the iteration until $i=d$.
    
    Let $t=\lfloor d/2 \rfloor$. Then each $u\in A_t$ has degree at least $d-t\geq d/2$ in $G_t$, and hence in particular $|B_t|\geq d/2$. Therefore, $|A_t|\geq d/2$. Let $k=\lfloor \log(2n/d)\rfloor+1=\lfloor \log(8n/d)\rfloor-1$ and let $s=\lfloor t/\log(8n/d)\rfloor$. Note that $(k+1)s\leq t$, so $1+ks\leq t$ (else, if $s=0$, the statement of our lemma is trivial). Since $|A_1|\leq n$, $|A_{1+ks}|\geq |A_t|\geq d/2$ and $2^k\geq 2n/d$, there exists some $j\geq 1$ with $|A_{j+s}|\geq |A_j|/2$. Let $H$ be the graph whose vertex set is $A_j\cup B_j$ and whose edge set is the union of the matchings $M_j, M_{j+1},\dots, M_{j+s}$. Each of these matchings has size at least $|A_{j+s}|\geq |A_j|/2$, so $e(H)\geq (s+1)|A_j|/2$. Thus, $H$ has average degree at least $(s+1)/2$. Then $H$ contains a non-empty subgraph $H'$ with minimum degree at least $(s+1)/4$. On the other hand, clearly $H$ has maximum degree at most $s+1$ and therefore so does $H'$. It follows that $H'$ is $4$-almost-regular with minimum degree at least $(s+1)/4\geq \frac{t}{4\log(8n/d)}\geq \frac{d}{25\log(8n/d)}$.
\end{proof}

\begin{proof}[ of Theorem \ref{thm:reg}]
By Theorem \ref{thm:regular in almost regular}, there exists some $C_0\geq 1$ such that every $4$-almost-regular graph with average degree at least $C_0r$ contains an $r$-regular subgraph. Let $G$ be an $n$-vertex graph with average degree $d\geq 100C_0r\log(n/r)$. By Proposition \ref{prop:almostreg sharp}, $G$ has a $4$-almost-regular subgraph $H$ with average degree at least $\frac{d}{100\log(32n/d)}$.
Note that
$$\frac{d}{100\log(32n/d)}\geq \frac{100C_0r\log(n/r)}{100\log(32n/d)}\geq \frac{100C_0r\log(n/r)}{100\log(n/r)}=C_0r,$$
where the second inequality used that $d\geq 32r$.
Hence, $H$ contains an $r$-regular subgraph by the choice of $C_0$.
\end{proof}

\section{Graphs without dense regular subgraphs} \label{sec:lower bounds}

In this section we prove our lower bounds, namely Propositions \ref{prop:constructiondsmall} and \ref{prop:lower bound d large}.

For Proposition \ref{prop:constructiondsmall}, it suffices to prove the following result.

\begin{prop} \label{prop:constructiondsmall adjusted}
    There exists some $c>0$ such that the following holds. For any positive integers $r$ and $n$ with $c^{-1}\leq r\leq c\log n$, there exists a bipartite graph with at least $n$ but at most $2n$ vertices and with average degree at least $\frac{1}{400}r^2\log(\frac{\log n}{r})$, which does not contain an $r$-regular subgraph.
\end{prop}

To see that Proposition \ref{prop:constructiondsmall adjusted} indeed implies Proposition \ref{prop:constructiondsmall}, note that for $r<c^{-1}$ Proposition \ref{prop:constructiondsmall} follows from Theorem~\ref{thm:PRSorigconstruction}, whereas for $c\log n<r\leq \frac{1}{2}\log n$, we can apply Proposition \ref{prop:constructiondsmall adjusted} with $r=\lfloor c\log n \rfloor$ and use that an $r$-regular bipartite graph has an $s$-regular subgraph for every $s\leq r$ (and in each case, we select an $n$-vertex induced subgraph uniformly at random).

\begin{proof}[ of Proposition \ref{prop:constructiondsmall adjusted}]
    Let $c$ be a sufficiently small positive constant. Let $i_{\textrm{min}}=\frac{1}{20}r\log r+r$ and let $i_{\textrm{max}}=\frac{1}{20}r\log \log n-r$. Let $A$ be a set of size $n$ and for each $i_{\textrm{min}}\leq i\leq i_{\textrm{max}}$, let $B_i$ be a set of size $n/2^{2^{20i/r}}$ such that all these sets are pairwise disjoint. Consider a random bipartite graph $G$ with parts $A$ and $\bigcup_{i=i_{\textrm{min}}}^{i_{\textrm{max}}} B_i$, where each vertex in $A$ is joined to precisely $\lfloor r/8 \rfloor$ random vertices of $B_i$, for each $i_{\textrm{min}}\leq i\leq i_{\textrm{max}}$.

    Note that for each $i_{\textrm{min}}\leq i\leq i_{\textrm{max}}-1$, we have \begin{equation}        |B_i|/|B_{i+1}|=2^{2^{20(i+1)/r}}/2^{2^{20i/r}}=2^{2^{20(i+1)/r}-2^{20i/r}}=2^{2^{20i/r}(2^{20/r}-1)}\geq 2^{2^{20i/r}\cdot \frac{10}{r}}\geq 2^{2^{\log r}\cdot \frac{10}{r}}= 2^{10}. \label{eqn:ratio}
    \end{equation}
    In particular, $|B_{i+1}|\leq |B_i|/2$, so $\sum_{i=i_{\textrm{min}}}^{i_{\textrm{max}}} |B_i|\leq n$. Hence, $G$ has at most $2n$ vertices. The number of edges in $G$ is precisely $n\cdot (i_{\textrm{max}}-i_{\textrm{min}}+1)\cdot \lfloor r/8 \rfloor\geq \frac{1}{200}nr^2(\log \log n-\log r-100)=\frac{1}{200}nr^2(\log(\frac{\log n}{r})-100)$. Since $r\leq c\log n$ and $c$ is sufficiently small, we have $100\leq \frac{1}{2}\cdot \log (\frac{\log n}{r})$, therefore $e(G)\geq \frac{1}{400}nr^2\log(\frac{\log n}{r})$.
    It follows that the average degree of $G$ is at least $\frac{1}{400}r^2\log(\frac{\log n}{r})$.

    It remains to prove that with positive probability $G$ does not contain an $r$-regular subgraph. This will follow shortly from the following claim.

    \medskip

    \noindent \emph{Claim.}
        Let $m$ be a positive integer no greater than $2n$. Let $j$ be the largest integer in the interval $[i_{\textrm{min}},i_{\textrm{max}}]$ such that $|B_j|\geq m$ (if no such $j$ exists, then let $j=i_{\textrm{min}}$). Then the probability that there exists a set $S\subset A\cup \bigcup_{i=i_{\textrm{min}}}^{j-1} B_i$ of size at most $m$ such that $G[S]$ contains at least $mr/8$ edges is at most $2^{-m}$.

    \medskip

    \noindent \emph{Proof of Claim.}
        If $j=i_{\textrm{min}}$, then there is nothing to prove, so let us assume that $j\geq i_{\textrm{min}}+1$. We will use the union bound. The number of ways to choose $S$ is at most $\sum_{\ell=1}^m \binom{2n}{\ell}\leq \sum_{\ell=1}^m \binom{4n}{\ell}\leq m\binom{4n}{m}\leq m(4en/m)^m\leq (8en/m)^m$. The number of ways to choose which $\lceil mr/8\rceil$ edges should be present in $G$ is at most $\binom{m^2}{\lceil mr/8\rceil}\leq (em^2/\lceil mr/8\rceil)^{\lceil mr/8\rceil}\leq (8em/r)^{\lceil mr/8\rceil}$. Moreover, the probability that these $\lceil mr/8\rceil$ edges are present in $G$ is at most $(\lfloor r/8\rfloor/|B_{j-1}|)^{\lceil mr/8\rceil}\leq (r/|B_{j-1}|)^{\lceil mr/8\rceil}$. Indeed, the individual probability that such an edge is present is at most $\lfloor r/8\rfloor/|B_{j-1}|$, and conditional on some edges being present, this probability can only decrease. Hence, the probability that there exists a set $S\subset A\cup \bigcup_{i=i_{\textrm{min}}}^{j-1} B_i$ of size at most $m$ such that $G[S]$ contains at least $mr/8$ edges is at most
        \begin{align*}
            (8en/m)^m\cdot (8em/r)^{\lceil mr/8\rceil}\cdot (r/|B_{j-1}|)^{\lceil mr/8\rceil}
            &=(8en/m)^m\cdot (8em/|B_{j-1}|)^{\lceil mr/8\rceil}\leq (8en/|B_j|)^m\cdot (8e|B_j|/|B_{j-1}|)^{\lceil mr/8\rceil} \\
            &\leq (8en/|B_j|)^m\cdot (8e|B_j|/|B_{j-1}|)^{mr/8}=\left( \frac{8en}{|B_j|}\cdot \left(\frac{8e|B_j|}{|B_{j-1}|}\right)^{r/8}\right)^m.
        \end{align*}
        By (\ref{eqn:ratio}), we have $|B_j|/|B_{j-1}|\leq 2^{-2^{20(j-1)/r}\cdot \frac{10}{r}}$, so
        $$\frac{n}{|B_j|}\cdot \left(\frac{|B_j|}{|B_{j-1}|}\right)^{r/8}\leq 2^{2^{20j/r}}\cdot 2^{-2^{20(j-1)/r}\cdot \frac{5}{4}}\leq 2^{2^{20j/r}}\cdot 2^{-2^{20j/r}\cdot \frac{6}{5}}=2^{-\frac{1}{5}\cdot 2^{20j/r}}\leq 2^{-\frac{1}{5}\cdot 2^{\log r+20}}\leq 2^{-10r}.$$ It follows that $\frac{8en}{|B_j|}\cdot (\frac{8e|B_j|}{|B_{j-1}|})^{r/8}\leq 1/2$, which completes the proof of the claim. \hfill $\boxdot$

    \medskip
    
    We will now show that if for all $1\leq m\leq 2n$ and every set $S\subset A\cup \bigcup_{i=i_{\textrm{min}}}^{j-1} B_i$ of size at most $m$, $G[S]$ has fewer than $mr/8$ edges (with $j$ depending on $m$ as in the claim), then $G$ does not contain an $r$-regular subgraph. (This is sufficient to complete the proof of the proposition, using the claim and the union bound.) 

    Indeed, for the sake of contradiction, assume that $G$ contains an $r$-regular subgraph $H$. Let $m=|H|$. Note that $e(H)= mr/2$. Let $j$ be defined as in the claim, and let $S=V(H)\cap(A\cup \bigcup_{i=i_{\textrm{min}}}^{j-1} B_i)$. By the assumption on $G$, $H[S]$ has fewer than $mr/8$ edges. 
    Furthermore, by the construction of $G$, each vertex in $A$ sends at most $r/8$ edges to $B_j$ and to $B_{\min(j+1,i_{\textrm{max}})}$. Hence, the number of edges in $H$ incident to $\bigcup_{i=i_{\textrm{min}}}^{\min(j+1,i_{\textrm{max}})} B_i$ is less than $mr/8+2\cdot m\cdot r/8=3mr/8$. This leads to a contradiction if $j\ge i_{\textrm{max}} - 1$.
    Otherwise, it follows that $H$ has at least $e(H)/4$ edges incident to $\bigcup_{i=j+2}^{i_{\textrm{max}}} B_i$. On the other hand, by the definition of $j$, we have $|B_{j+1}|<m$ and hence, using (\ref{eqn:ratio}), $|\bigcup_{i=j+2}^{i_{\textrm{max}}} B_i|\leq 2|B_{j+2}|\leq m/10$. 
    Therefore, as $H$ is $r$-regular, the number of edges in $H$ incident to $\bigcup_{i=j+2}^{i_{\textrm{max}}} B_i$ is at most $\frac{m}{10}\cdot r< e(H)/4$. This is a contradiction, and completes the proof.
\end{proof}

Most cases of Proposition \ref{prop:lower bound d large} follow from the following result.

\begin{prop}
 \label{prop:lower bound d large adjusted}
    There is some $c>0$ such that for all positive integers $r,n\geq 2$ with $20\log n\leq r\leq n/100$, there exists an $n$-vertex graph with average degree at least $cr\log(n/r)$ which does not have an $r$-regular subgraph.
\end{prop}

\begin{proof}
    Let $\ell=\frac{1}{2}\log(n/r)$, and let $A$, $B_1$, $B_2$, \dots, $B_{\ell}$ be disjoint sets with $|A|=n/2$ and $|B_i|=2^i r$ for all $i\in [\ell]$. We will define a random bipartite graph $G$ with parts $A$ and $\bigcup_{i=1}^{\ell} B_i$ which with positive probability will satisfy the following properties simultaneously: (i) $G$ has at least $\frac{1}{800} nr\log(n/r)$ edges and (ii) $G$ has no $r$-regular subgraph.
    Since $|\bigcup_{i=1}^{\ell} B_i|\leq 2^{\ell+1}r\leq 2(n/r)^{1/2}r\leq n/2$, the graph $G$ has at most $n$ vertices, so (after adding isolated vertices to $G$ if necessary) we obtain a suitable $n$-vertex graph.

    For each $i\in [\ell]$, let each edge between $A$ and $B_i$ be present independently at random with probability $p_i=\frac{1}{100\cdot 2^i}$.
    Now, note that
    \begin{equation*}
        \mathbb{E}[e(G)]=\sum_{i=1}^{\ell} \sum_{u\in A} \sum_{v\in B_i} p_i=\sum_{i=1}^{\ell} \sum_{u\in A} r/100=\frac{1}{200} \ell n r=\frac{1}{400} nr\log(n/r)
    \end{equation*}
    and
    \begin{equation*}
        \textrm{Var}(e(G))=\sum_{i=1}^{\ell} \sum_{u\in A} \sum_{v\in B_i} p_i(1-p_i)\leq \sum_{i=1}^{\ell} \sum_{u\in A} \sum_{v\in B_i} p_i=\mathbb{E}[e(G)],
    \end{equation*}
    so, by Chebysev's inequality, $\mathbb{P}\left(e(G)<\frac{1}{800}nr\log(n/r)\right)=o(1)$.

    It remains to prove that $G$ does not contain an $r$-regular subgraph with probability $1-o(1)$. This will be shown using the following claim.

    \noindent \emph{Claim.} Let $m$ be a positive integer. Let $j$ be the largest integer in $[\ell]$ such that $|B_j|\leq m/4$, and if no such integer exists, then let $j=0$. Then the probability that there exist $A'\subset A$ and $B'\subset \bigcup_{i=j+1}^{\ell} B_i$ such that $|A'|\leq m$, $|B'|\leq m$ and $e_G(A',B')> mr/2$ is at most $n^{-m}$.

    \medskip

    \noindent \emph{Proof of Claim.} The statement is trivial for $j=\ell$, so let us assume that $j<\ell$. Fix some $A'\subset A$ and $B'\subset \bigcup_{i=j+1}^{\ell} B_i$ such that $|A'|\leq m$, $|B'|\leq m$. Note that if $v\in B'$, then $v\in B_i$ for some $i\geq j+1$, so $p_i\leq \frac{1}{100\cdot 2^{j+1}}$. By the definition of $j$, we have $2^{j+1}r=|B_{j+1}|>m/4$, so $p_i\leq \frac{4r}{100m}\leq \frac{r}{4m}$. Thus, applying a Chernoff bound (Lemma~\ref{chernoff}),
    \[
    \mathbb{P}(e_G(A',B')> mr/2)\leq \mathbb{P}(\bin(m^2,r/4m)> mr/2)\leq \exp(-mr/8).
    \]
    The number of choices for the pair $(A',B')$ is at most $(\sum_{i=1}^m \binom{n}{i})^2\leq n^{2m}$, so by the union bound the probability that there exist $A'\subset A$ and $B'\subset \bigcup_{i=j+1}^{\ell} B_i$ such that $|A'|\leq m$, $|B'|\leq m$ and $e_G(A',B')> mr/2$ is at most $n^{2m}\exp(-mr/8)\leq n^{-m}$, where for the last inequality we used that $r\geq 20\log n$. This completes the proof of the claim. \hfill $\boxdot$

    \medskip
    
    We will now show that if for all positive integers $m$ there exist no $A'\subset A$ and $B'\subset \bigcup_{i=j+1}^{\ell} B_i$ such that $|A'|\leq m$, $|B'|\leq m$ and $e_G(A',B')> mr/2$ (with $j$ depending on $m$ as in the claim), then $G$ does not contain an $r$-regular subgraph. (This is sufficient to complete the proof of the proposition, using the claim and the union bound.)

    Indeed, for the sake of contradiction, assume that $G$ does contain an $r$-regular subgraph $H$. Let $A'=V(H)\cap A$, let $m=|A'|$ and let $B'=V(H)\cap \bigcup_{i=j+1}^{\ell} B_i$. Note that, by the regularity of $H$, we have $|B'|\leq |V(H)\cap \bigcup_{i=1}^{\ell} B_i|=|V(H)\cap A|=m$. Moreover, since $H$ is $r$-regular and (by the definition of $j$) $|\bigcup_{i=1}^{j} B_i|< m/2$, we have $e_H(A',B')> e(H)-\frac{m}{2}r= mr/2$, which contradicts the assumption about $G$.
\end{proof}

To obtain the remaining cases of Proposition \ref{prop:lower bound d large} which do not follow from Proposition \ref{prop:lower bound d large adjusted} (i.e. where $\frac{1}{2}\log n\leq r< 20\log n$), we can apply Proposition \ref{prop:constructiondsmall} to obtain an $n$-vertex graph $G_0$ with average degree $\Omega(\log^2 n)$ and no $s$-regular subgraph for $s=\frac{1}{2}\log n$. Let $G$ be a bipartite subgraph of $G_0$ containing at least half of the edges of $G_0$. Then $G$ is a suitable construction since any $r$-regular subgraph of $G$ would have an $s$-regular subgraph for $s=\frac{1}{2}\log n$.

\section{Concluding remarks} \label{sec:concluding remarks}

Recall that $d(r,n)$ is defined to be the smallest $d$ such that every $n$-vertex graph with average degree at least $d$ contains an $r$-regular subgraph. In this paper we proved that
$$
\left\{\begin{array}{lr}
        d(r,n)=\Theta(r\log(n/r)), & \text{for } \frac{1}{2}\log n\leq r\leq n/2\\
        \Omega(r^2 \log(\frac{\log n}{r}))\leq d(r,n)\leq O(r^2\log\log n), & \text{for } 3\leq r<\frac{1}{2}\log n
        \end{array}\right\}
$$

It is natural to make the following conjecture.

\begin{conj} \label{conj:transition}
    There exists some $C$ such that if $r\leq \frac{1}{2}\log n$, then every $n$-vertex graph with average degree at least $Cr^2 \log(\frac{\log n}{r})$ contains an $r$-regular subgraph.
\end{conj}

The difficulty in proving Conjecture \ref{conj:transition} using the proof method of Theorem \ref{thm:general ES} is that when $r\approx \log n$, then one can construct graphs in which typical edges have an endvertex of large enough degree so that we cannot apply the strategy of \textbf{case 1}, but where this degree is not large enough to apply the strategy of \textbf{case 2} (see Section~\ref{sec:tight ES}). More precisely, if our host graph has average degree $r^2 \log(\frac{\log n}{r})$, then in the proof of Theorem \ref{thm:general ES} we cannot choose $t_0$ to be larger than $r\log(\frac{\log n}{r})$. However, this means that in order to be able to deal with \textbf{case 2}, we would need to strengthen Lemma~\ref{lem:small jump implies k-regular} so that the required lower bound on $d$ is only about $2^{t_0}\approx (\frac{\log n}{r})^r$, which is significantly smaller than $r^{3r}(\log n)^2$ when $r$ is close to $\log n$. A natural approach to achieve this would be strengthening Lemma~\ref{lemma:num_matching} by allowing $t$ to be larger, but there are examples of almost-regular $r$-uniform hypergraphs with no matchings covering significantly more than $1/r$ proportion of the vertices, showing that $t$ cannot be taken to be larger in that lemma.

\textbf{Regularising hypergraphs.} Some of our techniques can also be used to (nearly-)regularise hypergraphs. In particular, an analogue of Theorem \ref{thm:constant diff} can be proved for linear hypergraphs by modifying the proof of Lemma~\ref{lem:one step with aux} as follows. We split the vertex set of our $k$-uniform linear hypergraph $G$ into the set $U_L$ of `low-degree' vertices and the set $U_H$ of `high-degree' vertices as before, and we let $G'$ be the subhypergraph obtained from $G$ by applying the following random edge and vertex deletions for some appropriate small $\eps>0$.

\begin{itemize}
    \item[(a)] Each $e\in E(G)$ is deleted independently with probability $1-(1-\eps)^{s_e}$, where $s_e=|e\cap U_H|$.
    \item[(b)] Each $v\in U_L$ is deleted independently with probability $\eps$.
\end{itemize}

Note that in the graph ($k=2$) case we recover the random deletions applied in the proof of Lemma \ref{lem:one step with aux}.
Now, let $e\in E(G)$. Observe that $e$ is deleted by (a) with probability $1-(1-\eps)^{s_e}$, and it is deleted by (b) with probability $1-(1-\eps)^{k-s_e}$ (since each of the $k-s_e$ vertices in $e\cap U_L$ survives with probability $1-\eps$). Hence, the probability that $e\in E(G')$ is precisely $(1-\eps)^k$. However, for any $u\in U_L$, if we condition on the event that $u\in V(G')$, then each hyperedge containing $u$ survives with a greater probability, $(1-\eps)^{k-1}$. Furthermore, if our hypergraph is linear, then, for each $u\in V(G)$, conditional on $u\in V(G')$, the events that $e\in E(G')$ are independent over all $e\in E(G)$ containing $u$. Hence, the degrees of low-degree vertices will drop slower than the degrees of high-degree vertices, and we can expect the hypergraph to become more regular.

Unfortunately, there is no known analogue of Corollary \ref{cor:AFK} for hypergraphs. Thus, we are unable to extend Theorem~\ref{thm:regular in almost regular} to this setting, and instead make the following conjecture.

\begin{conj} \label{conj:hypergraph}
    For each positive integer $k\geq 2$ and $\lambda\geq 1$, there exists some $c=c(k,\lambda)>0$ such that every $k$-uniform linear hypergraph with average degree $d$ and maximum degree at most $\lambda d$ contains an $r$-regular subhypergraph for some $r\geq cd$.
\end{conj}

\textbf{Acknowledgements.} We thank the anonymous referees for their helpful comments.

\bibliographystyle{abbrv}
\bibliography{references}

\end{document}